\documentclass[11pt]{article}
\usepackage{amssymb}
\usepackage{latexsym}
\usepackage{amsfonts}
\usepackage{amsmath}
\newtheorem{con0}{Theorem}[section]
\newtheorem{con1}[con0]{Condition}
\newtheorem{thrm}{Theorem}[section]
\newtheorem{lemma}[thrm]{Lemma}

\newtheorem{remark}[thrm]{Remark}

\numberwithin{equation}{section}
\usepackage[dvips]{color}
\usepackage{mathtools}
\mathtoolsset{showonlyrefs}
\usepackage{mathrsfs}
\usepackage{comment}
\usepackage{hyperref}
\usepackage{xcolor}

\def\proof{\noindent{\it Proof.~}}
\def\qed{\hfill$\square$\smallskip}

\def\bgcondition{\begin{con1}}\def\edcondition{\end{con1}}

\makeatletter
\begin{document}
	\allowdisplaybreaks

	\title{\Large \bf Small value probabilities of additive and derivative martingales in supercritical branching  Brownian motions and super Brownian motions}
\author{ \bf 
	Shukai Chen\footnote{The research of this author is supported by the National Key R\&D Program of
		China (No. 2022YFA1006003), NSFC grant of China (No. 12401167), Fujian
		Provincial Natural Science Foundation of China (No. 2024J08050), and the
		Education and Scientific Research Project for Young and Middle-aged Teachers in
		Fujian Province of China (No. JAT231015).
	}
	\hspace{1mm}\hspace{1mm} and \hspace{1mm}\hspace{1mm}
	Haojie Hou
	}
	\date{}
	\maketitle
	
	\begin{abstract}
		 In this paper, we establish asymptotics for the small value probabilities of additive and derivative martingales in both supercritical branching Brownian motions and super Brownian motions, thereby extending the corresponding results for Galton--Watson processes and continuous-state branching processes. For the derivative martingale in branching Brownian motion, our result also agrees with the findings in the arXiv version of Arguin et al. [arXiv:1008.4386 v1] and with those of Hu [Ann. Inst. H. Poincar\'e Probab. Stat., 2016].
	\end{abstract}
	
	\medskip
	

	\medskip
	

\section{Introduction and Main result}

The goal of this paper is to study the small value probabilities of certain martingales arising in supercritical branching Brownian motions and super Brownian motions. Let $\mathcal{B}_b^+(\mathbb{R})$ be the space of  non-negative bounded Borel functions on $\mathbb{R}$, and let $\mathcal{M}_F(\mathbb{R})$ be the space of finite measures on $\mathbb{R}$, equipped with the topology of weak convergence. For each $f \in \mathcal{B}_b^+(\mathbb{R})$ and $\mu \in \mathcal{M}_F(\mathbb{R})$, we write $\langle f, \mu\rangle$ for the integral of $f$ with respect to $\mu$.

\subsection{Branching Brownian motion}

A supercritical branching Brownian motion (BBM) is a continuous-time Markov branching process defined as follows. Initially, a single particle is located at $x$ and moves according to a standard Brownian motion $((B_t)_{t \geq 0}, \mathbf{P}_x)$. After an independent exponential time with parameter $\beta > 0$, the particle dies and gives birth to $k$ offspring with probability $p_k$. The offspring obey the same branching rule and evolve independently from their birth sites. The law of this process is denoted by $\mathbb{P}_x$ and we set $\mathbb{P}:= \mathbb{P}_0$ for simplicity. For each $t > 0$, let $Z_t$ be the counting measure formed by all particles alive at time 
$t$.
The so-called generating function is defined by
$$
f(s)=\sum_{k=0}^\infty p_k s^k,\quad |s|\le1.
$$
Let $q_c$ be the smallest root of $f(s)=s$ for $s\in[0,1]$. In this paper, we focus on the supercritical case, i.e., $m := \sum_{k=0}^\infty k p_k \in (1, \infty)$. In this case, it is well known that $q_c\in(0,1)$.
For each $\lambda \in \mathbb{R}$, define the additive martingale
\[
W_t(\lambda) := e^{-(\beta(m-1) + \frac{1}{2}\lambda^2)t} \langle e^{-\lambda \cdot}, Z_t \rangle.
\]
The convergence of $W_t(\lambda)$ has been widely studied over the past several decades; see, for example, \cite{Biggins, Chauvin,  Kingman, K04, Neveu}. 
We introduce the following assumption:
\bgcondition\label{cond1}
$|\lambda|< \sqrt{2\beta (m-1)} \ \mbox{and}\  \sum_{k=1}^\infty k(\log k) p_k<\infty. $
\edcondition

\noindent
As summarized in Kyprianou \cite[Theorem 1]{K04}, $W_t(\lambda)$ converges in $L^1(\mathbb{P})$ to a random variable $W_\infty(\lambda)$ if and only if Condition \ref{cond1} holds.

For $\lambda\in \mathbb{R}$, define
the derivative martingale
\[
\partial W_t(\lambda) := e^{-(\beta(m-1) + \frac{1}{2}\lambda^2)t} \langle (\cdot + \lambda t) e^{-\lambda \cdot}, Z_t \rangle.
\]
When $|\lambda| \geq \sqrt{2\beta(m-1)}$, Kyprianou \cite{K04} and Ren and Yang \cite{YR11} proved that $\partial W_t(\lambda)$ converges $\mathbb{P}$-almost surely to a non-negative limit $\partial W_\infty(\lambda)$, and the limit is non-trivial if and only if 
\bgcondition\label{cond2}
$|\lambda|= \sqrt{2\beta (m-1)} \  \mbox{and}\ \sum_{k=1}^\infty k(\log k)^2 p_k<\infty. $
\edcondition

\noindent
For $|\lambda| < \sqrt{2\beta(m-1)}$, it remains open (see \cite[Remark 4]{K04}) whether there exists a necessary and sufficient condition for the $L^1(\mathbb{P})$-convergence of $\partial W_t(\lambda)$. A sufficient condition 
in this regime
for the $L^1(\mathbb{P})$-convergence is $\sum_{k=1}^\infty k (\log k)^{\alpha} p_k < \infty$ for some $\alpha > 3/2$; see Hou et al. \cite[Proposition 2.6]{HRS2024}.

We now 
recall some known results on the small value probabilities of the aforementioned martingales.
Since $W_t(\lambda) \stackrel{\mathrm{d}}{=} W_t(-\lambda)$ and $\partial W_t(\lambda) \stackrel{\mathrm{d}}{=} - \partial W_t(-\lambda)$ under $\mathbb{P}$ for $\lambda \in \mathbb{R}$, 
it suffices to consider the case $\lambda \geq 0$. When $\lambda = 0$, Harris \cite[Remark, p.~110]{Harris1963} (see also \cite{Dubuc, MO2008} for the discrete-time case) showed that
the limit
$\lim_{\varepsilon \to 0+} \varepsilon^{-\kappa_0} \mathbb{P}(W_\infty(0) \in (0, \varepsilon))$ exists for $\kappa_0 := (1 - f'(q_c))/(m - 1)$, where $q_c$ is the extinction probability. 
When $\lambda \in (0, \sqrt{2\beta(m-1)})$, Liu \cite[Theorem 2.4]{Liu2001} proved that $\mathbb{P}(W_\infty(\lambda)
\in (0,\varepsilon)
) \leq C \varepsilon^a$ for some constants $C, a > 0$ as $\varepsilon \to 0+$.

For the derivative martingale $\partial W_t(\lambda)$, 
when $\lambda \in (0,\sqrt{2\beta(m-1)})$, since  $\partial W_t(\lambda)$
is a signed 
martingale  in this case, it is more natural to consider the left tail of $\partial W_\infty(\lambda)$ and  
the asymptotic behavior of $\mathbb{P}
(\partial W_\infty(\lambda)< -x)$ as $x \to +\infty$ was given by Chen et al. \cite[Theorem 1.1]{CHM2025} for the branching Wiener process. For the case $\lambda = \sqrt{2\beta(m-1)}$, Hu \cite[Theorem 1.3]{Hu2016} proved that $\mathbb{P}(\partial W_\infty(\sqrt{2\beta(m-1)}) \in (0, \varepsilon)) \asymp \varepsilon^{\kappa_1}$ as $\varepsilon \to 0+$ for  some $\kappa_1 > 0$. In an arXiv version of Arguin et al. \cite[Proposition A.1]{ABK10}, using a PDE argument, they proved that when $p_2 = 1$
and $\beta=1$, 
$
\mathbb{P}
(\partial W_\infty(\sqrt{2\beta(m-1)}) \in (0, \varepsilon)) \sim C_1 \varepsilon^{\kappa_1}$ for some $C_1 > 0$
and $\kappa_1= \sqrt{2}-1$.

Related small values probabilities have also been studied for other models. For the branching process in the B\"ottcher case, see  \cite{BB1993, FW2009, Hambly}; for branching process with immigration, see \cite{CLR2014, Sidorova}; for multitype branching process case, see \cite{Chu2014, Jones}.

Our main goal in the BBM setting is to establish precise asymptotic behavior for $\mathbb{P}(W_\infty(\lambda) \in (0, \varepsilon))$ when $\lambda \in (0, \sqrt{2\beta(m-1)})$, and for $\mathbb{P}(\partial W_\infty(\lambda) \in (0, \varepsilon))$ when $\lambda = \sqrt{2\beta(m-1)}$, as $\varepsilon \to 0+$.

\subsection{Super Brownian motion}
We next introduce the model of super Brownian motion. Let $\psi$ be a function of the form:
\begin{equation}\label{branching mechan}
		\psi(\lambda) = -\alpha \lambda + 
        \beta_0\lambda^2 + \int_0^\infty \left(e^{-\lambda y} - 1 + \lambda y\right) \nu(\mathrm{d}y), \quad \lambda \ge 0,
\end{equation}
where $\alpha > 0$, $
\beta_0\ge 0$, and $\nu$ is a $\sigma$-finite measure satisfying
\[
\int_0^\infty (y^2 \wedge y) \, \nu(\mathrm{d}y) < \infty.
\]
The function $\psi$ is called a {\it branching mechanism}. Our standard assumption is the following Grey's condition:
\begin{align}\label{Grey's-condition}
	\psi(+\infty) = +\infty \quad \text{and} \quad \int^\infty \frac{1}{\psi(\xi)} \, \mathrm{d}\xi < \infty.
\end{align}
Since $\psi$ is strictly convex with $\psi(0) = 0$ and $\psi'(0+) < 0$, there are exactly two roots in $[0, \infty)$: one is $0$, and the other, denoted by $\lambda^*$, is strictly positive.  
The super Brownian motion $X = ((X_t)_{t \geq 0}, \mathbb{P}_\mu)$ for each $\mu \in \mathcal{M}_F(\mathbb{R})$ with branching mechanism $\psi$ is an $\mathcal{M}_F(\mathbb{R})$-valued Markov process such that for all $\phi \in \mathcal{B}_b^+(\mathbb{R})$ and $t \geq 0$,
\[
-\log \mathbb{E}_{\mu}\left(e^{-\langle \phi, X_t \rangle}\right) = \langle u_{\phi}(t, \cdot), \mu \rangle, \quad \mu \in \mathcal{M}_F(\mathbb{R}),
\]
where $u_{\phi}(t, x)$ is the unique positive solution to the equation
\[
u_{\phi}(t, x) + \mathbf{E}_x \int_0^t \psi(u_{\phi}(t - s, B_s)) \, \mathrm{d}s = \mathbf{E}_x(\phi(B_t)).
\]
Note that the integral equation is equivalent to the following PDE:
	\begin{equation*} 
		\frac{\partial}{\partial t} u_{\phi}(t, x) - \frac{1}{2} \frac{\partial^2}{\partial x^2} u_{\phi}(t, x) = -\psi(u_{\phi}(t, x)), \quad t > 0, \; x \in \mathbb{R},
\end{equation*} 
with initial condition $u_{\phi}(0, x) = \phi(x)$.

We denote $\mathbb{P}:=\mathbb{P}_{\delta_0}$ for simplicity. For each $\lambda\in \mathbb{R}$, the additive martingale and derivative martingale in super Brownian motion are defined by
\[
W_t^X(\lambda) := e^{-(\alpha + \frac{1}{2}\lambda^2)t} \langle e^{-\lambda \cdot}, X_t \rangle
\quad \text{and} \quad
\partial W_t^X(\lambda) := e^{-(\alpha + \frac{1}{2}\lambda^2)t} \langle (\cdot + \lambda t) e^{-\lambda \cdot}, X_t \rangle.
\]
According to Kyprianou et al. \cite{KLMR12},  $W_t^X(\lambda)$ converges in $L^1(\mathbb{P})$ to some random variable $W_\infty^X(\lambda)$ if and only if
\bgcondition\label{cond3}
$|\lambda| < \sqrt{2\alpha}$ and $\displaystyle \int_{[1,\infty)} r (\log r) \, \nu(\mathrm{d}r) < \infty$.
\edcondition

\noindent
Also, for the convergence of the derivative martingale, we introduce the following condition
\bgcondition\label{cond2.5}
$\displaystyle \int^\infty \frac{1}{\sqrt{\int_{\lambda^*}^\xi \psi(u) \, \mathrm{d}u}} \, \mathrm{d}\xi < \infty$.
\edcondition

\noindent
Obviously, Condition \ref{cond2.5} is stronger than \eqref{Grey's-condition}.
For $|\lambda| \geq \sqrt{2\alpha}$, $\partial W_t^X(\lambda)$ converges $\mathbb{P}$-almost surely to a non-negative limit $\partial W_\infty^X(\lambda)$, and the limit is non-degenerate if and only if
\bgcondition\label{cond4}
$|\lambda| = \sqrt{2\alpha}$ and $\displaystyle \int_{[1,\infty)} r (\log r)^2 \, \nu(\mathrm{d}r) < \infty$.
\edcondition

\noindent
When $\lambda = 0$, Bingham \cite{Bingham} proved that the limit $\lim_{\varepsilon \to 0+} \varepsilon^{-\rho} \mathbb{P}_{\delta_0}(W_\infty^X(0) \in (0, \varepsilon))$ exists for $\rho = \psi'(\lambda^*)/\alpha$. For continuous-state branching process with immigration, see \cite{CLR2012}. 
To the best of our knowledge, the small value probabilities of $W_\infty^X(\lambda)$ for $\lambda \in (0, \sqrt{2\alpha})$ and of $\partial W_\infty^X(\sqrt{2\alpha})$ remain unknown, and this constitutes another main goal of our paper.

\subsection{Main result}

For branching Brownian motion, set
\begin{align*}	
	M_\infty(\lambda) := 
	\begin{cases}
		W_\infty(\lambda), & \lambda \in (0, \sqrt{2\beta(m-1)}), \\[4pt]
		\partial W_\infty(\sqrt{2\beta(m-1)}), & \lambda = \sqrt{2\beta(m-1)}.
	\end{cases}
\end{align*}
Similarly, for super Brownian motion, set
\begin{align*}	
	M_\infty^X(\lambda) := 
	\begin{cases}
		W_\infty^X(\lambda), & \lambda \in (0, \sqrt{2\alpha}), \\[4pt]
		\partial W_\infty^X(\sqrt{2\alpha}), & \lambda = \sqrt{2\alpha}.
	\end{cases}
\end{align*}
Our main result is as follows. 
\begin{thrm}\label{main result}
(i) Let $\lambda \in (0, \sqrt{2\beta(m-1)}]$. Assume either Condition \ref{cond1} or Condition \ref{cond2} holds. Then
\[
\lim_{\varepsilon \to 0+} \varepsilon^{-\rho} \, \mathbb{P}\bigl( M_\infty(\lambda) \in (0, \varepsilon) \bigr) \quad \text{exists and lies in } (0, \infty),
\]
where
\[
\rho := \frac{1}{\lambda} \left( \sqrt{ \left( \frac{\lambda}{2} + \frac{\beta(m-1)}{\lambda} \right)^2 + 2\beta \bigl(1 - f'(q_c)\bigr) } - \left( \frac{\lambda}{2} + \frac{\beta(m-1)}{\lambda} \right) \right).
\]

(ii) Let $\lambda \in (0, \sqrt{2\alpha}]$. Assume Condition \ref{cond2.5} holds and either Condition \ref{cond3} or Condition \ref{cond4} holds. Then
\[
\lim_{\varepsilon \to 0+} \varepsilon^{-\rho^X} \, \mathbb{P}\bigl( M_\infty^X(\lambda) \in (0, \varepsilon) \bigr) \quad \text{exists and lies in } (0, \infty),
\]
where
\[
\rho^X := \frac{1}{\lambda} \left( \sqrt{ \left( \frac{\lambda}{2} + \frac{\alpha}{\lambda} \right)^2 + 2 \psi'(\lambda^*) } - \left( \frac{\lambda}{2} + \frac{\alpha}{\lambda} \right) \right).
\]	
\end{thrm}

\begin{remark}
Since
\[
\rho = \frac{2\beta (1 - f'(q_c))}{
	\sqrt{ (\lambda^2/2 + \beta(m-1))^2 + 2\beta (1 - f'(q_c)) \lambda^2 } + \lambda^2/2 + \beta(m-1)
}
\]
and
\[
\rho^X = \frac{2 \psi'(\lambda^*)}{
	\sqrt{ (\lambda^2/2 + \alpha)^2 + 2 \psi'(\lambda^*) \lambda^2 } + \lambda^2/2 + \alpha},
\]
letting $\lambda \to 0+$, we obtain
\[
\rho \to \frac{1 - f'(q_c)}{m - 1}, \qquad
\rho^X \to \frac{\psi'(\lambda^*)}{\alpha}.
\]
These limits coincide with the classical results of Harris \cite{Harris1963} for Galton--Watson processes and Bingham \cite{Bingham} for continuous-state branching processes, respectively.	
\end{remark}

\section{Proof of the main result}

Our proof strategy is based on the following well-known result (see, for example, \cite[Theorem 1.7.1', p.~38]{BGT}): for any positive random variable $X$ and $\rho \geq 0$,
\begin{align}\label{Lap-small}
	\lim_{\theta \to +\infty} \theta^\rho \, \mathbb{E}(e^{-\theta X}) = C \in (0, \infty) 
	\quad \Longleftrightarrow \quad
	\lim_{\varepsilon \to 0+} \varepsilon^{-\rho} \, \mathbb{P}(X \leq \varepsilon) = \frac{C}{\Gamma(1+\rho)}.
\end{align}

We will prove the two cases together. Define
\[
\gamma:= \beta(m-1)\ \mbox{in BBM,} \quad \mbox{and}\quad 
\gamma := \alpha\ \mbox{in SBM}.
\]
Set $v := \lambda/2 + \gamma/\lambda$. 

For BBM, 
recall that $q_c \in [0,1)$ is
the extinction probability; then $f(q_c) = q_c$. From \cite[Theorem 1.2]{BK1997} (for additive martingale) and \cite[Proposition A.3(iii)]{Aidekon} (for derivative martingale), we see that $\mathbb{P}(M_\infty(\lambda) = 0) = q_c$.  
It was shown in \cite{K04}  that for $\lambda \in (0, \sqrt{2\gamma}]$, the function $\Phi_v(x) := \mathbb{E}\bigl(\exp\{-e^{-\lambda x} M_\infty(\lambda)\}\bigr)$ solves the following ordinary differential equation:
\begin{align*}
	\frac{1}{2}\Phi_v'' + v \Phi_v' + \beta(f(\Phi_v) - \Phi_v) = 0
, \quad  \Phi_v(-\infty)=q_c,\quad \Phi_v(+\infty)=1.
\end{align*}
A direct calculation gives
\begin{align}\label{Laplace-BBM}
	\mathbb{E}\left(\exp\left\{-e^{-\lambda x} M_\infty(\lambda)\right\} \,\Big|\, M_\infty(\lambda) > 0\right) = \frac{\Phi_v(x) - q_c}{1 - q_c}
	.
\end{align}

Define
\[
Q_v(x) := \frac{\Phi_v(-x) - q_c}{1 - q_c}.
\]
Then $Q_v$ is decreasing in $x$ with $Q_v(+\infty) = 0$, $Q_v(-\infty) = 1$, and $Q_v$ satisfies
\begin{align}\label{ODE2}
	\frac{1}{2} Q_v'' - v Q_v' + \beta\left( \frac{f((1 - q_c)Q_v + q_c) - q_c}{1 - q_c} - Q_v \right) = 0.
\end{align}

For SBM, by Kyprianou et al. \cite{KLMR12}, for each $\lambda \in (0, \sqrt{2\gamma}]$, the function
\[
\Phi_v(x) := -\log \mathbb{E}\left(\exp\left\{-e^{-\lambda x} M_\infty^X(\lambda)\right\}\right)
\]
solves the equation
\begin{align*}
	\frac{1}{2}\Phi_v'' + v \Phi_v' - \psi(\Phi_v) = 0.
\end{align*}
Moreover, $\Phi_v$ is decreasing with $\Phi_v(-\infty) = \lambda^*$ and $\Phi_v(+\infty) = 0$. 
Observe that as $x \to -\infty$, 
\begin{align}\label{Laplace-SBM}
	\mathbb{E}
	\left(\exp\left\{-e^{-\lambda x} M_\infty^X(\lambda)\right\} \,\Big|\, M_\infty^X(\lambda) > 0\right)
	& = \frac{e^{-\lambda^*}\left(e^{\lambda^* - \Phi_v(x)} - 1\right)}{1 - e^{-\lambda^*}} \nonumber\\
	&\sim \frac{e^{-\lambda^*}\bigl(\lambda^* - \Phi_v(x)\bigr)}{1 - e^{-\lambda^*}}
	.
\end{align}

Define
\[
Q_v(x) := 1 - \frac{\Phi_v(-x)}{\lambda^*}
.
\]
Then
 $Q_v$ is decreasing in $x$ with $Q_v(+\infty) = 0$, $Q_v(-\infty) = 1$, and $Q_v$ satisfies
\begin{align}\label{ODE2'}
	\frac{1}{2} Q_v'' - v Q_v' + 
	\frac{1}{\lambda^*}\psi(\lambda^*(1-Q_v))= 0.
\end{align}

For $z\in [0,1]$, set 
\begin{align*}
	G(z) :=  
	\left\{
	\begin{aligned}
		&\beta\big( \frac{f((1-q_c)z+q_c)-q_c}{1-q_c} -z\big), &  \mbox{in BBM},\\
		&\psi(\lambda^*(1-z))/\lambda^*, & \mbox{in SBM}.
	\end{aligned}
	\right.
\end{align*}
Combining \eqref{ODE2} and \eqref{ODE2'}, $Q_v$ is a decreasing function on $\mathbb{R}$ with $Q_v(+\infty) = 0$, $Q_v(-\infty) = 1$, and
\begin{align}\label{ODE3}
	\frac{1}{2} Q_v'' - v Q_v' + G(Q_v) = 0.
\end{align}
Note that $G(0) = 0$ and 
\begin{align}\label{Def-of-G-prime}
	G'(0+) = 
	\begin{cases}
		\beta\bigl(f'(q_c) - 1\bigr), & \text{in BBM},\\[4pt]
		-\psi'(\lambda^*), & \text{in SBM}.
	\end{cases}
\end{align}
Define
\begin{align}\label{Def-of-R}
	R(z) := 
	\begin{cases}
		0, & z = 0,\\[4pt]
		\dfrac{G(z) - z G'(0+)}{z}, & 0 < z \le 1.
	\end{cases}
\end{align}
\begin{lemma}\label{lem1}
The function $R$ is increasing on $[0,1]$. Moreover, there exists a constant $K > 0$ such that $0 \leq R(z) \leq K z$ for all $z \in [0,1]$.
\end{lemma}
\proof 
According to the definition of $G$, we have
\begin{align*}
	G''(z) := 
	\begin{cases}
		\beta(1 - q_c) \displaystyle \sum_{k=2}^\infty k(k-1)p_k\bigl((1 - q_c)z + q_c\bigr)^{k-2}, & \text{in BBM},\\[8pt]
		2\beta_0\lambda^*+\lambda^* \displaystyle \int_0^\infty y^2 e^{-\lambda^*(1 - z)y} \, \nu(\mathrm{d}y), & \text{in SBM}.
	\end{cases}
\end{align*}
and 
\begin{align*}
	G'''(z) := 
	\begin{cases}
		\beta(1 - q_c)^2 \displaystyle \sum_{k=3}^\infty k(k-1)(k-2)p_k \bigl((1 - q_c)z + q_c\bigr)^{k-3}, & \text{in BBM},\\[8pt]
		(\lambda^*)^2 \displaystyle \int_0^\infty y^3 e^{-\lambda^*(1 - z)y} \, \nu(\mathrm{d}y), & \text{in SBM}.
	\end{cases}
\end{align*}
Therefore, $0 < G''(z)$ for all $z \in [0,1]$, and $G''(z) \leq G''(1/2)$ for all $z \in [0, 1/2]$. This implies that for all $z \in (0, 1/2)$,
\begin{align}\label{e1}
	R(z) 
	&= \frac{1}{z} \int_0^z \bigl(G'(y) - G'(0+)\bigr) \, \mathrm{d}y 
	= \frac{1}{z} \int_0^z \int_0^y G''(u) \, \mathrm{d}u \, \mathrm{d}y \nonumber\\
	&\leq \frac{1}{z} \int_0^z \int_0^y G''(1/2) \, \mathrm{d}u \, \mathrm{d}y 
	= \frac{z}{2} G''(1/2).
\end{align}
Since $R(0) = 0$, the above inequality also holds for $z = 0$. For $z \in [1/2, 1]$, since $G$ is continuous on $[0,1]$, we have
\begin{align}\label{e2}
	R(z) \leq 2 \sup_{t \in [0,1]} |G(t)| + |G'(0+)| 
	\leq 2z \left( 2 \sup_{t \in [0,1]} |G(t)| + |G'(0+)| \right).
\end{align}
Combining \eqref{e1} and \eqref{e2} yields the second result.

For the first result, it follows from the fact that $G''(z) > 0$ that $G$ is strictly convex, which implies that for any $0 < y < z \leq 1$,
\begin{align*}
	R(z) 
	&= \frac{G(z)}{z} - G'(0+) 
	= \frac{1}{y} \left( \frac{y}{z} G(z) + \frac{z-y}{z} G(0) \right) - G'(0+) \nonumber\\
	&> \frac{G(y)}{y} - G'(0+) = R(y),
\end{align*}
as desired.
\qed

Recall that $\mathbf P_x$ denotes the law of a standard Brownian motion starting from $x$. 
For each $\theta\in \mathbb{R}$, let $\mathbf P_x^{-\theta}$ be the law of the Brownian motion with drift $-\theta$ starting from $x$. 
By the Feynman--Kac formula, for any $t > 0$ and $x > 0$, it follows from \eqref{ODE3} that
\begin{align}
	Q_v(x) 
	&= \mathbf E^{-v}_x \left( Q_v(B_t) \exp\left\{ \int_0^t R(Q_v(B_s)) \, \mathrm{d}s \right\} e^{G'(0+) t} \right) \nonumber\\
	&= \mathbf E^{-v}_x \left( Q_v(B_t) \exp\left\{ \int_0^t R(Q_v(B_s)) \, \mathrm{d}s \right\} e^{-|G'(0+)| t} \right),
\end{align}
where in the last equality we used the fact that $G'(0+) < 0$ by \eqref{Def-of-G-prime}.
Define
\begin{equation}\label{theta_0}
	\theta_0^2 := v^2 + 2 |G'(0+)| > v^2.
\end{equation}
It follows from Girsanov's theorem that, for $x > 0$,
\begin{align*}
	Q_v(x) 
	&= \mathbf E^{-\theta_0}_x
	\left( Q_v(B_t) \exp\left\{ \int_0^t R(Q_v(B_s)) \, \mathrm{d}s \right\} e^{(\theta_0 - v)(B_t - x)} \right) \nonumber\\
	&= e^{-(\theta_0 - v)x} \mathbf E^{-\theta_0}_x
	\left( Q_v(B_t) \exp\left\{ \int_0^t R(Q_v(B_s)) \, \mathrm{d}s \right\} e^{(\theta_0 - v) B_t} \right).
\end{align*}
Therefore, under $\mathbf{P}_x^{-\theta_0}$, the process
\[
Q_v(B_t) \exp\left\{ \int_0^t R(Q_v(B_s)) \, \mathrm{d}s \right\} e^{(\theta_0 - v) B_t}
\]
is a martingale. Define $\tau_y := \inf\{ t \ge 0 : B_t = y \}$. Then, by the optional stopping theorem, we obtain
\begin{align}\label{eq:2.9}
	Q_v(x) = e^{-(\theta_0 - v)x} \mathbf E^{-\theta_0}_x
	\left( Q_v(B_{t \wedge \tau_0}) \exp\left\{ \int_0^{t \wedge \tau_0} R(Q_v(B_s)) \, \mathrm{d}s \right\} e^{(\theta_0 - v) B_{t \wedge \tau_0}} \right).
\end{align}
Our next lemma states that one may replace $t \wedge \tau_0$ by $\tau_0$ in \eqref{eq:2.9}.

\begin{lemma}\label{le:2.1}
For any $x > 0$, we have
\[
Q_v(x) e^{(\theta_0 - v)x}
= Q_v(0) \, \mathbf E^{-\theta_0}_x
\left( \exp\left\{ \int_0^{\tau_0} R(Q_v(B_s)) \, \mathrm{d}s \right\} \right)
=: F_\lambda(x).
\]
\end{lemma}

\proof
Recall from \eqref{ODE3} that $x \mapsto Q_v(x)$ is decreasing and bounded between $0$ and $1$.
Then, by Lemma \ref{lem1},
\begin{equation}\label{W(B)<W(0)}
	R(Q_v(B_s)) \le R(Q_v(0)) < R(1) 
	\quad \text{and} \quad 
	Q_v(B_t) < 1, \qquad \forall t \ge 0, \; \forall s \le \tau_0.
\end{equation}
Combining \eqref{theta_0}, \eqref{W(B)<W(0)}, and Girsanov's theorem, we obtain
\begin{align}\label{ineq:2.10}
	& \mathbf E^{-\theta_0}_x
	\left( Q_v(B_t) \exp\left\{ \int_0^{t} R(Q_v(B_s)) \, \mathrm{d}s \right\} e^{(\theta_0 - v) B_t} \mathbf{1}_{\{t < \tau_0\}} \right) \nonumber \\
	&\le e^{R(Q_v(0)) t} \, \mathbf E^{-\theta_0}_x
	\left( e^{(\theta_0 - v) B_t} \mathbf{1}_{\{t < \tau_0\}} \right) \nonumber \\
	&=
	e^{\theta_0x}
	e^{(R(Q_v(0)) - \frac{1}{2}\theta_0^2) t} \, \mathbf E_x
	\left( e^{-v B_t} \mathbf{1}_{\{t < \tau_0\}} \right) \nonumber \\
	&\le 
	e^{\theta_0x}
	e^{(R(Q_v(0)) - \frac{1}{2}\theta_0^2) t} \, \mathbf P_x(t < \tau_0).
\end{align}
Recall from the proof of Lemma \ref{lem1} that $G$ is strictly convex. Since $G(0) = G(1) = 0$, we have $G(z) < 0$ for all $z \in (0,1)$. Therefore, combining \eqref{Def-of-R} and \eqref{theta_0}, we obtain
\begin{equation}\label{ineq:2.13}
	R(Q_v(0)) - \frac{\theta_0^2}{2} < -G'(0+) - \frac{\theta_0^2}{2} = -\frac{v^2}{2}.
\end{equation}
Hence, \eqref{ineq:2.10} yields
\begin{equation*}
	\mathbf E^{-\theta_0}_x
	\left( Q_v(B_t) \exp\left\{ \int_0^{t} R(Q_v(B_s)) \, \mathrm{d}s \right\} e^{(\theta_0 - v) B_t} \mathbf{1}_{\{t < \tau_0\}} \right)
	\leq 
	e^{\theta_0x}
	\mathbf P_x(t < \tau_0).
\end{equation*}
As a consequence, it follows from \eqref{eq:2.9} that
\begin{align}\label{ineq:2.14}
	& \left| Q_v(x) e^{(\theta_0 - v)x} - Q_v(0) \, \mathbf E^{-\theta_0}_x \left[ \exp\left\{ \int_0^{\tau_0} R(Q_v(B_s)) \, \mathrm{d}s \right\} \mathbf{1}_{\{\tau_0 \leq t\}} \right] \right| \nonumber\\
	&
	\leq 
	e^{\theta_0x}
	\mathbf P_x(\tau_0 > t).
\end{align}
By the reflection principle for Brownian motion,
\[
\mathbf P_x(\tau_0 > t) = 1 - 2 \mathbf P_x(B_t \leq 0)
= 1 - \frac{\sqrt{2}}{\sqrt{\pi}} \int_{-\infty}^{-x/\sqrt{t}} e^{-z^2/2} \, \mathrm{d}z
\to 0
\]
as $t \to \infty$. Substituting this into \eqref{ineq:2.14} completes the proof of the lemma.
\qed

\bigskip 
Now we are ready to give the proof of Theorem \ref{main result}.
\bigskip 

\noindent{\it Proof of Theorem \ref{main result}}. 
{\bf Step 1}. 
Recall the definition of $F_\lambda$ from Lemma \ref{le:2.1}. In this step, we prove that
\begin{equation}\label{main eq}
	\lim_{x \to \infty} F_\lambda(x) =: F_\lambda(\infty) > 0, \qquad \forall \lambda \in (0, \sqrt{2\gamma}].
\end{equation}
Since the function $R$ is non-negative by Lemma \ref{lem1}, for any $x, y \ge 0$ and $\lambda \in (0, \sqrt{2\gamma}]$,
\begin{align}\label{ineq:2.17}
	F_\lambda(x + y)
	&= Q_v(0) \, \mathbf E_{x+y}^{-\theta_0} \left( \exp\left\{ \int_0^{\tau_0} R(Q_v(B_s)) \, \mathrm{d}s \right\} \right) \nonumber \\
	&\ge Q_v(0) \, \mathbf E_{x+y}^{-\theta_0} \left( \exp\left\{ \int_{\tau_y}^{\tau_0} R(Q_v(B_s)) \, \mathrm{d}s \right\} \right) \nonumber \\
	&= Q_v(0) \, \mathbf E_y^{-\theta_0} \left( \exp\left\{ \int_0^{\tau_0} R(Q_v(B_s)) \, \mathrm{d}s \right\} \right) = F_\lambda(y),
\end{align}
where in the second equality we used the strong Markov property. It then follows straightforwardly from \eqref{ineq:2.17} that, for all $x \ge 0$, $F_\lambda(x) \ge F_\lambda(0) = Q_v(0) > 0$, which implies \eqref{main eq}.

{\bf Step 2}. In this step, we prove that
\begin{equation}\label{main eq2}
	F_\lambda(\infty) < \infty, \qquad \forall \lambda \in (0, \sqrt{2\gamma}].
\end{equation}
Combining \eqref{eq:2.9}, \eqref{W(B)<W(0)}, and Girsanov's theorem, for any $x > 0$ and $t \ge 0$,
\begin{align*}
	Q_v(x)
	 e^{(\theta_0 - v)x}
	&= \mathbf E_x^{-\theta_0}
	\left[
	Q_v(B_{t \wedge \tau_0}) \exp\left\{ \int_0^{t \wedge \tau_0} R(Q_v(B_s)) \, \mathrm{d}s \right\}
	e^{(\theta_0 - v) B_{t \wedge \tau_0}}
	\right] \nonumber \\
	&\le \mathbf E_x
	\left[
	\exp\left\{ R(Q_v(0)) (t \wedge \tau_0) \right\}
	e^{(\theta_0 - v) B_{t \wedge \tau_0}}
	e^{-\theta_0 (B_{t \wedge \tau_0} - x) - \theta_0^2 (t \wedge \tau_0)/2}
	\right] \nonumber \\
	&= e^{\theta_0 x} \, \mathbf E_x
	\left[ e^{(R(Q_v(0)) - \theta_0^2/2)(t \wedge \tau_0)} e^{-v B_{t \wedge \tau_0}} \right]
	=: e^{\theta_0 x} \, \mathbf E_x
	\left[ e^{-c_0 (t \wedge \tau_0)} e^{-v B_{t \wedge \tau_0}} \right].
\end{align*}
Since $B_{t \wedge \tau_0} \ge 0$ and $c_0 > v^2/2$ by \eqref{ineq:2.13}, it follows from the dominated convergence theorem that
\begin{align}
	Q_v(x)
	 &\le e^{v x} \lim_{t \to \infty} \mathbf E_x
	\left[ e^{-c_0 (t \wedge \tau_0)} e^{-v B_{t \wedge \tau_0}} \right] \nonumber \\
	&= e^{v x} \, \mathbf E_x \left[ e^{-c_0 \tau_0} \right]
	= e^{-(\sqrt{2c_0} - v)x} =: e^{-c_1 x}.
\end{align}

Since $c_1>0$,  let $A>0$ be a fixed constant such that $e^{-c_1 A} < \theta_0^2/(2K)$. Then combining Lemmas \ref{lem1} and  \ref{le:2.1}, and the strong Markov property, 
 for any $x>A$,
\begin{align}\label{F<H}
	F_\lambda(x)
	&= 
	F_\lambda(A) \mathbf E_x^{-\theta_0} \left( \exp\left\{ \int_0^{\tau_A} R(Q_v(B_s)) \, \mathrm{d}s \right\} \right) 
	\nonumber \\
	&\le 
	F_\lambda(A)  \mathbf E_x^{-\theta_0} \left( \exp\left\{ K \int_0^{\tau_A} e^{-c_1 B_s} \, \mathrm{d}s \right\} \right)\nonumber\\
	& = 	F_\lambda(A)  \mathbf E_{x-A}^{-\theta_0} \left( \exp\left\{ Ke^{-c_1 A} \int_0^{\tau_0} e^{-c_1 B_s} \, \mathrm{d}s \right\} \right)
	 =: H(x-A).
\end{align}
It is obvious that  for any $x\geq 0$,
\[
H(x)\leq F_\lambda(A)  \mathbf E_{x}^{-\theta_0} \left( e^{Ke^{-c_1 A}  \tau_0}\right) = 
F_\lambda(A)
\exp\left\{x\left(\theta_0 -\sqrt{\theta_0^2-2Ke^{-c_1 A}}\right)\right\}<\infty. 
\]
By the strong Markov property applied at $\tau_x$,
\begin{align*}
	H(2x)
	&= 
	F_\lambda(A)\mathbf E_{2x}^{-\theta_0} \left[ \exp\left\{ Ke^{-c_1A} \int_0^{\tau_x} e^{-c_1 B_s} \, \mathrm{d}s \right\} \right]	\cdot \mathbf E_x^{-\theta_0} \left[ \exp\left\{ K e^{-c_1A} \int_0^{\tau_0} e^{-c_1 B_s} \, \mathrm{d}s \right\} \right] \\
	&=
	\mathbf E_{2x}^{-\theta_0} \left[ \exp\left\{ Ke^{-c_1A} \int_0^{\tau_x} e^{-c_1 B_s} \, \mathrm{d}s \right\} \right] H(x).
\end{align*}
On the interval $[0, \tau_x]$, we have $B_s \ge x$, so $e^{-c_1 B_s} \le e^{-c_1 x}$. 
Therefore, for any $x\geq 0$, 
\begin{align*}
	&\mathbf E_{2x}^{-\theta_0} \left[ \exp\left\{ Ke^{-c_1A} \int_0^{\tau_x} e^{-c_1 B_s} \, \mathrm{d}s \right\} \right]	\le \mathbf E_{2x}^{-\theta_0} \left[ \exp\left\{ K e^{-c_1 (x+A)} \tau_x \right\} \right] \\	&= \exp\left\{ x \left( \theta_0 - \sqrt{\theta_0^2 - 2K e^{-c_1 (x+A)}} \right) \right\} = \exp\left\{ \frac{2K xe^{-c_1(x+A)}}{\theta_0+ \sqrt{\theta_0^2 - 2K e^{-c_1 (x+A)}} }\right\}   \\	&\le \exp\left\{ \frac{2K}{\theta_0} x e^{-c_1 (x+A)} \right\}.
\end{align*}
In conclusion,
 there exists $c_2 > 0$ such that for all 
 $x\geq 0$,
\begin{equation*}
	H(2x) \le H(x) \exp\left\{ c_2 e^{-c_1 x/2} \right\}.
\end{equation*}
This together with \eqref{F<H} implies that for any $n \in \mathbb{N}$,
\begin{align*}
	F_\lambda(2^n +A) 	&\le H(2^n ) 	\le H(1) \exp\left\{ c_2 \sum_{k=0}^{n-1} e^{-c_1 2^k  / 2} \right\} \\	&\le H(1) \exp\left\{ c_2 \sum_{k=0}^{\infty} e^{-c_1 2^k  / 2} \right\}.
\end{align*}
Taking $n \to \infty$ in the above inequality implies \eqref{main eq2}.

{\bf Step 3}. In this step, we prove the main theorem. For BBM, combining \eqref{Laplace-BBM}, Lemma \ref{le:2.1}, and the results obtained in the above two steps, we have
\begin{align*}
	&\lim_{x \to -\infty} e^{-(\theta_0 - v)x} \, \mathbb{E}\left( \exp\left\{ -e^{-\lambda x} M_\infty(\lambda) \right\} \,\Big|\, M_\infty(\lambda) > 0 \right) \\
	&= 
	\lim_{x \to -\infty} e^{-(\theta_0 - v)x} Q_v(-x)=F_\lambda(\infty),
\end{align*}
which implies, upon replacing $e^{-\lambda x}$ by $\theta$,
\[
\lim_{\theta \to +\infty} \theta^{\frac{\theta_0 - v}{\lambda}} \, \mathbb{E}\left( e^{-\theta M_\infty(\lambda)} \,\Big|\, M_\infty(\lambda) > 0 \right)
=
F_\lambda(\infty).
\]
Therefore, together with \eqref{Lap-small}, this yields
\begin{align*}
	\lim_{\varepsilon \to 0+} \varepsilon^{-\frac{\theta_0 - v}{\lambda}} \, \mathbb{P}\left( M_\infty(\lambda) \leq \varepsilon \,\big|\, M_\infty(\lambda) > 0 \right)
	= 
	\frac{1}{\Gamma(1 + (\theta_0-v)/\lambda)}
	F_\lambda(\infty),
\end{align*}
which is equivalent to
\[
\lim_{\varepsilon \to 0+} \varepsilon^{-\frac{\theta_0 - v}{\lambda}} \, \mathbb{P}\left( M_\infty(\lambda) \in (0, \varepsilon) \right)
= 
\frac{1-q_c}{\Gamma(1 + (\theta_0-v)/\lambda)} 
F_\lambda(\infty).
\]

The proof for SBM is similar. Combining \eqref{Laplace-SBM}, Lemma \ref{le:2.1}, and the results obtained in the above two steps, we have
\begin{align*}
	&\lim_{x \to -\infty} e^{-(\theta_0 - v)x} \, \mathbb{E}\left( \exp\left\{ -e^{-\lambda x} M_\infty^X(\lambda) \right\} \,\Big|\, M_\infty^X(\lambda) > 0 \right) \\
	&= \frac{e^{-\lambda^*} \lambda^*}{1 - e^{-\lambda^*}} F_\lambda(\infty).
\end{align*}
Repeating the same argument as for BBM, we conclude that
\[
\lim_{\varepsilon \to 0+} \varepsilon^{-\frac{\theta_0 - v}{\lambda}} \, \mathbb{P}\left( M_\infty^X(\lambda) \in (0, \varepsilon) \right)
= \frac{e^{-\lambda^*} \lambda^*}{\Gamma(1 + 
	(\theta_0-v)/\lambda
	)} F_\lambda(\infty).
\]
This completes the proof.
\qed

\bigskip
\noindent


	\begin{singlespace}
		
	\end{singlespace}

	\vskip 0.2truein
	\vskip 0.2truein

\bigskip

\noindent{\bf Shukai Chen:}  School of Mathematics and Statistics, Fujian Normal University,
Fuzhou 350117, Fujian, P. R. China Email: {\texttt skchen@fjnu.edu.cn}

\bigskip 
\noindent{\bf Haojie Hou:}  School of Statistics and Data Science, Nankai University, Tianjin 300071, P. R. China.  Email: {\texttt houhaojie@nankai.edu.cn}

\end{document}